\documentclass [12pt]{article}
\usepackage{graphicx,amssymb,amsfonts,latexsym,amsmath,amsthm,times}
\usepackage{epsfig}
\usepackage{fancyhdr}
\usepackage{color}
\usepackage[english]{babel}
\numberwithin{equation}{section}

\newtheorem{theorem}{Theorem}

\newtheorem{corollary}[theorem]{Corollary}

\newtheorem{definition}[theorem]{Definition}
\newtheorem{example}[theorem]{Example}

\newtheorem{lemma}[theorem]{Lemma}

\newcommand{\n}{\mathbb{N}}

\begin{document}
	
	\vspace*{2cm} \normalsize \centerline{\Large \bf Fixed Point Property in G-Complete Fuzzy Metric Space}
	
	\vspace*{1cm}
	
	\centerline{\bf Ismail TAHIRI$^a$ and Ahmed NUINO$^b$}
	
	\vspace*{0.5cm}
	
	\centerline{$^a$ Higher Normal School, AMCS Research Team, Teouan, Morocoo.}
	
	\centerline{$^b$ ISPITS Tetouan, AMCS Research Team, Teouan, Morocoo.}


	\vspace*{1cm}
	
	\noindent {\bf Abstract.}
	 Our purpose of this paper is to focus on fixed point property in fuzzy metric space. To achieve our objective, we will introduce a new contraction condition to examine the fixed point for multi-valued mapping, then we will be investigating the obtained result to ensure the existence and uniqueness of this property for single-valued mapping. To show the use of our main result, we will give the relative result in the ordinary metric space.

	\vspace*{0.5cm}
	
	\noindent {\bf Key words:}
	Fuzzy metric spaces; G-Cauchy sequence; $\zeta$-contraction multi-valued mapping; $\zeta$-contraction mapping; Fixed point.\\
	\noindent {\bf AMS subject classification 2020:} Primary 47H10; Secondary 47J26.

\section{Introduction}
The theory of fuzzy sets created by L. A. Zadeh\cite{emm17} has had considerable applications in several sciences \cite{emm18,emm19,emm20,emm21,emm22}, motivated Kramosil and Michalek \cite{emm23} to present fuzzy metric spaces. Later, in order to achieve a Hausdorff topology for it, George and Veeramani amended its concept \cite{emm24} which inspired the authors to study the fixed point property for single-valued mapping in this formulation of fuzzy metric space \cite{emm33,emm34,emm35}. Based on the collection of nonempty compact subsets of a particular fuzzy metric space, Rodr\'iguez-L\'opez and Romaguera create a Hausdorff fuzzy metric in the sense of George and Veeramani\cite{emm2}. Several fixed point outcomes have been attained for multi-valued mappings in fuzzy metric spaces \cite{emm26,emm27,emm28,emm30,emm31,emm32}.\\
This is the format of the paper. In Section $2$, we report some initial findings. In section $3$, we introduce a new $\zeta$-contraction multi-valued and single valued mapping, then we examine the fixed point property for this mappings types. In last section we give an application in ordinary metric space to highlight our main result.
\section{Preliminary}
\par In this section we recall briefly some well-known definitions and results in the fuzzy metric space theory.
\begin{definition}
A $\mathfrak{t}$-norm is a binary operation on $[0,1]$ such that
\begin{enumerate}
\item $\Gamma(\mu,\nu)=\Gamma(\nu,\mu)$,
\item $\Gamma(\mu,\Gamma(\nu,w))=\Gamma(\Gamma(\mu,\nu),w)$,
\item $\Gamma(\mu,\nu)<\Gamma(\mu,w)$ whenever $\nu<w$,
\item $\Gamma(\mu,1)=\Gamma(1,\mu)=\mu$.
\end{enumerate}
for all $\mu,\nu,w$ in $[0,1]$.
\end{definition}
\par For a basic examples of t-norms we introduce:
\begin{enumerate}
  \item $\Gamma_{L}(\mu,\nu)=\max\{\mu+\nu-1,0\}$,
  \item $\Gamma_{p}(\mu,\nu)=\mu\nu$,
  \item $\Gamma_{M}(\mu,\nu)=Min(\mu,\nu)$.
\end{enumerate}
For any $\mathfrak{t}$-norm $\Gamma$ we have $\Gamma\leq \Gamma_{M}$.\\
Let $\Gamma$ be a $\mathfrak{t}$-norm and $n\in\mathbb{N}$, for $(\mu_{1},\mu_{2},\ldots,\mu_{n})\in[0,1]^{n}$, the values $\Gamma(\mu_{1},\mu_{2},\ldots,\mu_{n})$ is defined by
\begin{center}
  $\Gamma_{i=1}^{0}\mu_{i}=1$, $\Gamma_{i=1}^{n}\mu_{i}=\Gamma(\Gamma_{i=1}^{n-1}\mu_{i},\mu_{n})=\Gamma(\mu_{1},\mu_{2},\ldots,\mu_{n})$.
\end{center}
Specially, we have
\begin{enumerate}
  \item $\Gamma_{L}(\mu_{1},\mu_{2},\ldots,\mu_{n})=\max(\sum_{i=1}^{n}\mu_{i}-(n-1),0)$,
  \item $\Gamma_{M}(\mu_{1},\mu_{2},\ldots,\mu_{n})=\min(\mu_{1},\mu_{2},\ldots,\mu_{n})$,
  \item $\Gamma_{p}(\mu_{1},\mu_{2},\ldots,\mu_{n})=\prod_{i=1}^{n}\mu_{i}$.
\end{enumerate}
and for $\mu\in[0,1]$ , we have
$$\Gamma^{n}(\mu)=\left\{\begin{array}{lll}
1 & \mbox{if } n=0,\\
\Gamma(\Gamma^{n-1}(\mu),\mu) & \mbox{ otherwise.}
\end{array}\right.$$
\begin{definition}
We said that a $\mathfrak{t}$-norm $\Gamma$ is of $\mathfrak{H}$-type if
 $$\forall\epsilon\in(0,1)\quad \exists\lambda\in(0,1):\rho>1-\lambda\Rightarrow \Gamma^{n}(\rho)>1-\epsilon \mbox{ for all } n\geq 1.$$
\end{definition}
\par Note that $\Gamma_{M}$ is an example of $\mathfrak{H}$-type $\mathfrak{t}$-norm, but $\Gamma_{L}$ is not and for more details see\cite{emm3}.

\begin{definition}
Let $\Xi$ be a non empty set. A fuzzy metric on $\Xi$ is a mapping $\Theta:\Xi\times \Xi\times(0,\infty)\longrightarrow(0,1]$ such that:
\begin{enumerate}
\item $\Theta(u,v,0)>0$,
\item $\Theta(u,v,\rho)=1$ if and only if $u=v$,
\item $\Theta(u,v,\rho)=\Theta(y,x,\rho)$,
\end{enumerate}
for all $u, v$ in $\Xi$ and $\rho>0$.
\end{definition}
\begin{definition}\cite{emm24}
A fuzzy metric space is a triple $(\Xi,\Theta,\Gamma)$ where $\Xi$ is a non empty set, $\Theta$ is a fuzzy metric on $\Xi$, $\Gamma$ is a continuous $\mathfrak{t}$-norm, and the following statements are satisfied for all  $u$, $v$, $z$ in $\Xi$ and $\rho,s>0$
\begin{enumerate}
\item $\Theta(u,z,\rho+s)\geq \Gamma(\Theta(u,v,\rho),\Theta(v,z,s))$,
\item $\Theta(u,v,.):(0,\infty)\rightarrow(0,1]$ is continuous.
\end{enumerate}
\end{definition}
\begin{example}\cite{emm24}\label{M31}
Let $(\Xi,d)$ be a metric space and $\Theta(u,v,\rho)=\displaystyle\frac{\rho}{\rho+d(u,v)}$ for all $u,v\in \Xi$ and $\rho>0$. Then $(\Xi,\Theta,\Gamma_{p})$ is a fuzzy metric space.
\end{example}
\begin{definition}
Let $\{x_{n}\}$ be a sequence in a fuzzy metric space $(\Xi,\Theta,\Gamma)$. We said that $\{x_{n}\}$
\begin{enumerate}
\item  converge to $x\in \Xi$ if for all $\epsilon\in(0,1)$ and $\rho>0$, there exists $N=N(\epsilon,\rho)\in\n$ such that $\Theta(x_{n},x,\rho)>1-\epsilon$ whenever $n\geq N$,
\item is a Cauchy sequence if for all $\epsilon\in(0,1)$ and $\rho>0$, there exists $N=N(\epsilon,\rho)\in\n$ such that $\Theta(x_{n},x_{m},\rho)>1-\epsilon$ whenever $n,m\geq N$,
\item is a G-Cauchy sequence if and only if for every positive integer $p$, for all $\epsilon\in(0,1)$ and $\rho>0$, there exists $N=N(\epsilon,\rho)\in\n$ such that $\Theta(x_{n},x_{n+p},\rho)>1-\epsilon$ whenever $n\geq N$.
\end{enumerate}
The space $(\Xi,\Theta,\Gamma)$ is complete (G-complete resp.) if every Cauchy (G-Cauchy resp.) sequence converges to some $x\in \Xi$.
\end{definition}
\par We note that every Cauchy sequence is a G-Cauchy sequence and there is a G-Cauchy sequence no Cauchy sequence wich is demonstrated by the following example \cite{emm24}.
\begin{example}
Let $\Xi=\mathbb{R}$ and $\Theta(u,v,\rho)=\displaystyle\frac{\rho}{\rho+|u-v|}$ for all $u,v\in \Xi$ and $\rho>0$. Then $(\Xi,\Theta,\Gamma_{p})$ is a fuzzy metric space. So the sequence $\{u_{n}\}$ which $u_{n}=\displaystyle\sum_{k=1}^{n}\frac{1}{k}$ is a G-Cauchy but no Cauchy.
\end{example}
It is believed that the fuzzy metric $\Theta$ is continuous on $\Xi\times \Xi\times(0,\infty)$ if $\lim_{n\rightarrow\infty}\Theta(u_{n},v_{n},\rho_{n})=\Theta(u,v,\rho)$ while $\{(u_{n},v_{n},\rho_{n})\}$ is a sequence in $\Xi\times \Xi\times(0,\infty)$ converges to a point $(u,v,\rho)\in \Xi\times \Xi\times(0,\infty)$.
\begin{lemma}\cite{emm2}
  If $(\Xi,\Theta,\Gamma)$ is a fuzzy metric space. Then the fuzzy metric $\Theta$ is continuous on $\Xi\times \Xi\times(0,\infty)$.
\end{lemma}
Think of $(\Xi,\Theta,\Gamma)$ as a fuzzy metric space. All non-empty compact subsets of $\Xi$ are represented by the letter $K(\Xi)$. We remember the idea of Hausdorff fuzzy metric that was produced by a fuzzy metric $\Theta$ in the following manner, in accordance with \cite{emm2}: For $A,B\in K(\Xi)$ and $\rho>0$ define
\begin{equation}\label{M10}
\Upsilon(A,B,\rho)=\min\{\inf_{v\in B}\Theta(A,v,\rho),\inf_{u\in A}\Theta(u,B,\rho)\}
\end{equation}
 where $\Theta(v,A,\rho)=\sup_{u\in A}\Theta(v,u,\rho)$. The triplet $(K(\Xi),\Upsilon,\Gamma)$ recognized as the Hausdorff fuzzy metric space.
\begin{lemma}\cite{emm2}\label{M8}
  Let $(\Xi,\Theta,\Gamma)$ be a fuzzy metric space. Then, for each $u\in \Xi$, $B\in K(\Xi)$ and $\rho>0$, there is $v\in B$ such that
  $\Theta(u,B,\rho)=\Theta(u,v,\rho)$.
\end{lemma}
\begin{definition}
Think of $(\Xi,\Theta,\Gamma)$ as a fuzzy metric space. A multi-valued mapping $S:\Xi\longrightarrow K(\Xi)$ is said possesses a fixed point if there exists $u\in \Xi$ such that $u\in Su$.
\end{definition}

\section{Main results}
\par The letter $\Pi$ stands for every function $\zeta:(0,\infty)\longrightarrow(0,\infty)$ verified  $\lim_{n\rightarrow\infty}\zeta^{n}(\rho)=0$ for all $\rho>0$.
\par Now we introduce the notion of $\zeta$-contraction multi-valued and single-valued mapping.
\begin{definition}
Think of $(\Xi,\Theta,\Gamma)$ as a fuzzy metric space. A multi-valued mapping $S:\Xi\longrightarrow K(\Xi)$ is called $\zeta$-contraction, if there is a $\zeta\in\Pi$ such that for all $u,v\in \Xi$ and $\rho>0$
  $$\Upsilon(Su,Sv,\zeta(\rho))\geq \Theta(u,v,\rho).$$
\end{definition}
If $\zeta(\rho)=kt$ we said that $S$ is contraction multi-valued mapping.
\begin{definition}
 Think of $(\Xi,\Theta,\Gamma)$ as a fuzzy metric space. A mapping $f:\Xi\longrightarrow \Xi$ is called $\zeta$-contraction, if there is a $\zeta\in\Pi$ such that for all $u,v\in \Xi$ and $\rho>0$
  $$\Theta(fu,fv,\zeta(\rho))\geq \Theta(u,v,\rho).$$
\end{definition}
\par Now we present some lemmas which plays an important role in the following.
\begin{lemma}\label{M19}
Think of $(\Xi,\Theta,\Gamma)$ as a fuzzy metric space such that $\lim_{\rho\rightarrow\infty}\Theta(u,v,\rho)=1$ for all $u,v\in \Xi$ and $\rho>0$. If $S:\Xi\longrightarrow K(\Xi)$ is a $\zeta$-contraction multi-valued mapping, then there is a sequence $\{u_{n}\}\subset \Xi$ satisfying
\begin{center}
$u_{n+1}\in Su_{n}$ and $\Theta(u_{n+1},u_{n+2},\zeta(\rho))\geq \Theta(u_{n},u_{n+1},\rho)$ for all $\rho>0$.
\end{center}
\end{lemma}
\begin{proof}
  Let $\rho>0$, $u_{0}\in \Xi$ and $u_{1}\in Su_{0}$, by lemma\ref{M8} there is $u_{2}\in Su_{1}$ such that
  \begin{align*}
    \Theta(u_{1},u_{2},\zeta(\rho)) &=  \Theta(u_{1},Su_{1},\zeta(\rho))\\
                           &\geq \mathcal{H}(Su_{0},Su_{1},\zeta(\rho)).
  \end{align*}
By induction we construct a sequence $\{u_{n}\}\subset \Xi$ such that
\begin{center}
$u_{n+1}\in Su_{n}$ and  $\Theta(u_{n},u_{n+1},\zeta(\rho))\geq \mathcal{H}(Su_{n-1},Su_{n},\zeta(\rho))$.
\end{center}
Since $S$ is $\zeta$-contractive, then $\Theta(u_{n},u_{n+1},\zeta(\rho))\geq \Theta(u_{n-1},u_{n},\rho)$, which is complete the proof.
\end{proof}
\begin{lemma}\label{M14}
  Let $\{u_{n}\}$ be a sequence in a fuzzy metric space $(\Xi,\Theta,\Gamma)$ under a $\mathfrak{t}$-norm of $\mathfrak{H}$-type where for every $u,v\in \Xi$ and $\rho>0$, $\lim_{\rho\rightarrow\infty}\Theta(u,v,\rho)=1$. If there is $\zeta\in\Gamma$ such that for all $n\in\n$ and $\rho>0$,
  \begin{equation}\label{M5}
  \Theta(u_{n+2},u_{n+1},\zeta(\rho))\geq \Theta(u_{n+1},u_{n},\rho),
  \end{equation}
  then $\{u_{n}\}\subset \Xi$ is a G-Cauchy sequence.
\end{lemma}
\begin{proof}
Let $\{u_{n}\}$ be a sequence satisfied the condition (\ref{M5}). It is obviously that for all $n\in\n$ and $\rho>0$
\begin{equation}\label{M1}
  \Theta(u_{n},u_{n+1},\zeta^{n}(\rho))\geq \Theta(u_{0},u_{1},\rho)
\end{equation}
Let $\delta\in(0,1)$ and $\lambda>0$.\\
The fact that $\Gamma$ is of $\mathfrak{t}$-type implies there exists $\epsilon\in(0,1)$ such that
\begin{equation}\label{M2}
  \rho>1-\epsilon \Longrightarrow \Gamma^{n}(\rho)>1-\delta\mbox{ for all } n\geq1.
\end{equation}
Since $\lim_{\rho\rightarrow\infty}\Theta(u_{0},u_{1},\rho)=1$, there exists $\rho_{0}>0$ such that
\begin{equation}\label{M3}
\Theta(u_{0},u_{1},\rho_{0})>1-\epsilon.
\end{equation}
Let $p$ be a positive integer.\\
As $\lim_{n\rightarrow\infty}\zeta^{n}(\rho_{0})=0$, there exists $N\in\n$ such that
\begin{equation}\label{M4}
\zeta^{n}(\rho_{0})<\displaystyle\frac{\lambda}{p} \mbox{ whenever } n\geq N.
\end{equation}
So, for $n\geq N$,
\begin{align*}
 \Theta(u_{n},u_{n+p},\lambda)&\geq \Gamma(\Theta(u_{n},u_{n+1},\zeta^{n}(\rho_{0})),\Theta(u_{n+1},u_{n+p},\lambda-\zeta^{n}(\rho_{0}))) \\
                        &\geq \Gamma(\Theta(u_{n},u_{n+1},\zeta^{n}(\rho_{0})),\Gamma(\Theta(u_{n+1},u_{n+2},\zeta^{n}(\rho_{0})), \\
                        &\hspace{5cm} \Theta(u_{n+2},u_{n+p},\lambda-2\zeta^{n}(\rho_{0})))).
\end{align*}
We concluded that
\begin{equation}\label{M29}
\begin{split} \Theta(u_{n},u_{n+p},\lambda) \geq \Gamma(\Theta(u_{n},u_{n+1},\zeta^{n}(\rho_{0})),\Gamma(\Theta(u_{n+1},u_{n+2},\zeta^{n}(\rho_{0})),\Gamma(\ldots,\\
   \Gamma(\Theta(u_{n+p-2},u_{n+p-1},\zeta^{n}(\rho_{0}),\Theta(u_{n+p-1},u_{n+p},\lambda-p\zeta^{n}(\rho_{0}))))\ldots).
\end{split}
\end{equation}
By (\ref{M29}), (\ref{M3}), (\ref{M1}) and (\ref{M2}) it follows that
\begin{equation}\label{M6}
\Theta(u_{n},u_{n+p},\lambda)\geq \Gamma^{p}(1-\epsilon).
\end{equation}
Using (\ref{M5}), we concluded that
\begin{equation}\label{M7}
 \Theta(u_{n},u_{n+p},\lambda) > 1-\delta.
\end{equation}
So $\{u_{n}\}$ is a G-Cauchy sequence.
\end{proof}
\begin{theorem}\label{M28}
Think of $(\Xi,\Theta,\Gamma)$ as a G-Complete fuzzy metric space under a $\mathfrak{t}$-norm of $\mathfrak{H}$-type such that for all $u,v\in \Xi$, $\lim_{\rho\rightarrow\infty}\Theta(u,v,\rho)=1$. If $S:\Xi\longrightarrow K(\Xi)$ is a $\zeta$-contraction multivalued mapping, then $S$ possesses a fixed point.
\end{theorem}
\begin{proof}
Since the multivalued mapping $S$ is $\zeta$-contraction, by lemma \ref{M19} there is a sequence $\{u_{n}\}\subset \Xi$ such that
\begin{center}
$u_{n+1}\in Su_{n}$ and $\Theta(u_{n+1},u_{n+2},\zeta(\rho))\geq \Theta(u_{n},u_{n+1},\rho)$ for all $\rho>0$.
\end{center}
By lemma\ref{M14} we concluded that $\{u_{n}\}$ is a G-Cauchy sequence. So it's converge to some $u\in \Xi$.\\
Now we shaw that $u\in Su$.\\
Let $\rho>0$, as $\lim_{n\rightarrow\infty}\zeta^{n}(\rho)=0$, we can choice a positive integer $N$ where
\begin{center}
  $\forall n\geq N$, $\zeta^{n}(\rho)<\rho$.
\end{center}
Let $n\geq N$, so $\Theta(u_{n+1},Su,\rho)\geq \Theta(Su_{n},Su,\zeta^{N+1}(\rho))$.\\
Since $S$ is $\zeta$-contraction, then $\Theta(u_{n+1},Su,\rho)\geq \Theta(u_{n},u,\zeta^{N}(\rho))$.\\
By lemma\ref{M8} there is $z\in Su$ satisfying $\Theta(u_{n+1},Su,\rho)=\Theta(u_{n+1},z,\rho)$.\\
So
\begin{equation}\label{M23}
   \Theta(u_{n+1},z,\rho)\geq \Theta(u_{n},u,\zeta^{N}(\rho)).
 \end{equation}
 Letting $n\longrightarrow\infty$ in (\ref{M23}), we have $\Theta(u,z,\rho)=1$, that mis $u=z$, which implies that $u\in Sx$.
\end{proof}
\begin{theorem}\label{M30}
Think of $(\Xi,\Theta,\Gamma)$ as a G-Complete fuzzy metric space under a $\mathfrak{t}$-norm of $\mathfrak{H}$-type such that for all $u,v\in \Xi$, $\lim_{\rho\rightarrow\infty}\Theta(u,v,\rho)=1$. If $f:\Xi\longrightarrow \Xi$ is a $\zeta$-contraction mapping, then $f$ possesses a unique fixed point $z$, moreover, the sequence $\{f^{n}u\}$ converges to $z$ for each $u\in \Xi$.
\end{theorem}
\begin{proof}
By Theorem \ref{M28} there exists a $z\in \Xi$ such that $fz=z$.\\
Now we demonstrate that $z$ is unique. For that we suppose that there is another $v\in \Xi$ where $fv=v$.\\
Let $\rho>0$. So $\Theta(fz,fv,\zeta(\rho))=\Theta(z,v,\zeta(\rho))\geq \Theta(z,v,\rho)$.\\
We check by induction that for all $n\in\n$, $\Theta(z,v,\zeta^{n}(\rho))\geq \Theta(z,v,\rho)$.\\
Since $\zeta^{n}(\rho)\longrightarrow 0$ as $n\longrightarrow\infty$, there is a positive integer $N$ such that $\zeta^{n}(\rho)<\rho$ for $n\geq N$.\\
So  $\Theta(z,v,\zeta^{n}(\rho))=\Theta(z,v,\rho)$.\\
We show that $\Theta(z,v,\rho)=1$ for all $\rho>0$.\\
Assume that there exists $\rho_{0}>0$ satisfied $\Theta(z,v,\rho)<1$, then there is $\rho_{1}>\rho_{0}$ such that $\Theta(z,v,\rho_{1})>\Theta(z,v,\rho_{0})$.\\
The fact that $\zeta^{n}(\rho)\longrightarrow 0$ as $n\longrightarrow\infty$ implied that there is a positive integer $n_{{0}}$ such that $\zeta^{n_{0}}(\rho_{1})<\rho_{0}$.\\
So $\Theta(z,v,\zeta^{n_{0}}(\rho_{1}))=\Theta(z,v,\rho_{1})\leq \Theta(z,v,\rho_{0})$. Which is a contradiction. Then we concluded that $z=v$.
\end{proof}
Noted that in the Proof of Theorem\ref{M28} the condition ”$\Theta_{pq}(0)=0$” was not used, then we have the Grabiec fixed point.
\begin{corollary}\cite{emm4}
Think of $(\Xi,\Theta,\Gamma)$ as a G-Complete fuzzy metric space under a $\mathfrak{t}$-norm of $\mathfrak{H}$-type such that for all $u,v\in \Xi$, $\lim_{\rho\rightarrow\infty}\Theta(u,v,\rho)=1$. If $f:\Xi\longrightarrow \Xi$ is a contraction mapping, then $f$ possesses a unique fixed point.
\end{corollary}
\section{Relative results in ordinary metric spaces}

\begin{theorem}
Let $(\Xi,d)$ be a metric space and $f:\Xi\longrightarrow \Xi$ be a mapping. We suppose that the following statement are fulfilled:
\begin{description}
  \item[a)] If $\{\iota_{n}\}$ is a sequence in $\Xi$ such that, for all $p\in\n$,\ $d(\iota_{n},\iota_{n+p})\longrightarrow 0$ as $n\longrightarrow\infty$, then $\{\iota_{n}\}$ is convergente;
  \item[b)] There exists $\zeta\in\Gamma$ satisfied
  \begin{center}
  $\rho d(f\iota,f\kappa)\leq\zeta(\rho)d(\iota,\kappa)$ for all $\iota,\kappa\in \Xi$ and $\rho>0$.
 \end{center}
\end{description}
Then $f$ possesses a unique fixed point, moreover, the sequence $\{f^{n}\iota\}$ converges to $z$ for each $\iota\in \Xi$.
\end{theorem}
\begin{proof}
For $\iota,\kappa\in \Xi$ and $\rho>0$, we put $\Theta(\iota,\kappa,\rho)=\displaystyle\frac{\rho}{\rho+d(\iota,\kappa)}$, so by example \ref{M31}, $(\Xi,\Theta,\Gamma_{p})$ is a fuzzy metric space satisfying $\lim_{\rho\rightarrow\infty}\Theta(\iota,\kappa,\rho)=1$ for all $\iota,\kappa\in \Xi$. Clearly the statement (a) implies that $(\Xi,\Theta,\Gamma_{p})$ is G-complete. To complete the proof, we show that the statement (b) ensures that $f$ is $\zeta$-contraction in $(\Xi,\Theta,\Gamma_{p})$.\\
In fact, let $\iota,\kappa\in \Xi$ and $\rho>0$, then
\begin{align*}
       d(f\iota,f\kappa)\leq\displaystyle\frac{\zeta(\rho)}{\rho}d(\iota,\kappa)& \Rightarrow & \rho\zeta(\rho)+td(f\iota,f\kappa)\leq\displaystyle \rho\zeta(\rho)+\zeta(\rho)d(\iota,\kappa) \\
        &\Rightarrow & \displaystyle\frac{1}{\rho\zeta(\rho)+td(f\iota,f\kappa)}\geq\displaystyle\frac{1}{\rho\zeta(\rho)+\zeta(\rho)d(\iota,\kappa)}\\
        &\Rightarrow & \displaystyle\frac{\zeta(\rho)}{\zeta(\rho)+d(f\iota,f\kappa)}\geq\displaystyle\frac{\rho}{\rho+d(\iota,\kappa)}\\
        &\Rightarrow & \Theta(f\iota,f\kappa,\zeta(\rho))\geq \Theta(\iota,\kappa,\rho).
     \end{align*}
So $f$ is $\zeta$-contraction in $(\Xi,\Theta,\Gamma_{p})$. By theorem \ref{M30}, $f$ possesses a unique fixed point $z$ and the sequence $\{f^{n}\iota\}$ converges to $z$ for all $\iota\in \Xi$.`
\end{proof}

	

\begin{thebibliography}{99}
		\bibitem{emm32} Y. Achtoun, M. L. Sefian and I. Tahiri, \textit{Muli-valued Hicks Contractions in b-Menger Spaces.} Nonlinear Functional Analysis and Applications, \textbf{29, 2} (2024), 477-485.
\bibitem{emm28} A. Amini-Harandi and F. Kiany, \textit{Fixed point and endpoint theorems for set-valued fuzzy contraction maps in fuzzy metric spaces.} Fixed Point Theory and Appl. \textbf{94} (2011), 9 pages.

\bibitem{emm27} A. Basit, A. Mujahid, S. Naeem and R. Zahid, \textit{Fixed points of Suzuki type generalized multivalued mappings in fuzzy metric spaces with applications.} Fixed Point Theory and Applications \textbf{36} (2015).
\bibitem{emm9} A. T. Bharucha-Reid and V. M. Sehgal, \textit{Fixed points of contraction mappings on probabilistic metric spaces.} Math Syst Theory \textbf{6} (1972), 97-102.

\bibitem{emm31} T. Došenovi\'ca, R. Dušan, B. Cari\'cb and S. Radenovi\'c, \textit{Multivalued generalizations of fixed point results in fuzzy metric spaces.} Nonlinear Analysis: Modelling and Control, \textbf{Vol. 21, No. 2} (2016), 1–12.

\bibitem{emm18} R. Fuller, \textit{Neural Fuzzy Systems.} Abo Akademis Tryckeri, Abo, ESF Series A: 443; 1995.

\bibitem{emm20} T. S. El Naschie, \textit{The two-slit experiment as the foundation of E-infinity of high energy physics.} Chaos Soliton Fractal \textbf{25} (2005), 509-514.

\bibitem{emm24} A. George and P. Veeramani, \textit{On some results in fuzzy metric spaces.} Fuzzy Sets Syst. \textbf{64} (1994), 395–399.
\bibitem{emm34} D. Gopal and C. Vetro, \textit{Some New Fixed Point Theorems in Fuzzy Metric Spaces.} Iranian Journal of Fuzzy systems, \textbf{11, 3} (2014), 95-107.
\bibitem{emm4} M. Grabiec, \textit{Fixed points in fuzzy metric spaces.} Fuzzy Sets and Systems \textbf{27} (1988), 385-389.
\bibitem{emm33} V. Gregori and A. Sapena, \textit{On Fixed-Point Theorems in Fuzzy Metric Spaces.} Fuzzy Sets and Systems, \textbf{125} (2002), 245-252.


\bibitem{emm3} O. Hadzi\`c, E. Pap and M. Budincevi\'c, \textit{Countable extension of triangular norms and their applications to the fixed point theory in probabilistic metric spaces.} Kybernetika \textbf{38} (2002), 363-382.
\bibitem{emm35} H. Huang, B. Cari\'c, T. Dosenovi\'c and M. Brdar, \textit{Fixed-Point Theorems in Fuzzy Metric Spaces via Fuzzy F-contraction.} Mathematics, \textbf{9 6} (2021), 641.

\bibitem{emm23} I. Kramosil and J. Michalek, \textit{Fuzzy metric and statistical metric spaces.} Kybernetica \textbf{11} (1975), 336-344.

\bibitem{emm19} A. McBratney and I. O. A Odeh, \textit{Application of fuzzy sets in soil science: fuzzy logic, fuzzy measurements and fuzzy decisions.} Geoderma \textbf{77} (19970, 85-113.

\bibitem{emm2} J. Rodríguez-López and S. Romaguera, \textit{The Hausdorff fuzzy metric on compact sets.} Fuzzy Sets Syst. \textbf{147} (2004), 273-283.

\bibitem{emm21} A. Sapena, S. Romaguera and P. Tirado, \textit{The Banach fixed point theorem in fuzzy quasi-metric spaces with
application to the domain of words.} Topol. Appl. \textbf{154} (2007), 2196–2203.
\bibitem{emm26} H. Shihuang and Q. Zheyong, \textit{Coupled fixed points for multivalued mappings in fuzzy metric spaces.} Fixed Point Theory and Applications \textbf{162} (2013).
\bibitem{emm30} A. Shoaib, A. Azam and A. Shahzad, \textit{Common fixed point results for the family of multivalued mappings satisfying contractions on a sequence in Hausdorff fuzzy metric spaces.} J. Computational Analysis and Applications \textbf{24 (4)} (2018), 692-699.

\bibitem{emm22} F. Steimann, \textit{On the use and usefulness of fuzzy sets in medical AI.} Artif. Intell. Med. \textbf{21} (2001), 131–137.

\bibitem{emm17} L. A. Zadeh, \textit{Fuzzy sets.} Inf. Control \textbf{8} (1965), 338-353.

\end{thebibliography}
\end{document}